%% file: main.tex
\documentclass[12pt]{article}

\input{macros}

\title{Digraphs of Large Girth and Dichromatic Number in Tournaments with Large Dichromatic Number}
\author{Pierre Charbit}
\author{Samuel Coulomb}
\affil{Universit\'e Paris Cit\'e, CNRS, IRIF, F-75013, Paris, France.}
\date{}

\begin{document}

\maketitle

\begin{abstract}
    In the 1960s, \Erd~and Hajnal conjectured that every graph with sufficiently large chromatic number contains a subgraph of large girth (size of a smallest cycle) and large chromatic number.
    In this paper, we prove that every tournament with sufficiently large dichromatic number contains a subdigraph of large digirth (size of a smallest directed cycle) and large dichromatic number.

    We investigate the same statement when replacing digirth by girth (of the underlying graph). We show that it implies the conjecture of \Erd~and Hajnal, and prove it for a particular family of tournaments.
\end{abstract}

\section{Introduction}

The \emph{girth} of a graph $G$, denoted $\girth(G)$, is the minimum size of a cycle in $G$.
In one of the first application of the probabilistic method to graph theory, \Erd~\cite{Erdos} established the existence of graphs with arbitrarily large girth and large chromatic number.
A few years later, he and Hajnal conjectured that moreover, every graph with large enough chromatic number contains a subgraph with large girth and large chromatic number.

\begin{conj}[\Erd--Hajnal \cite{ErdosHajnal}]\label{conj:EH}
    For every integers $k \ge 0$ and $\ell \ge 3$, there exists an integer $f(k,\ell) \ge 0$ such that every graph $G$ with $\chr(G) \ge f(k,\ell)$ contains a subgraph $H \subseteq G$ with $\chr(H) > k$ and $\girth(H) > \ell $.
\end{conj}

Since then, few progress has been made on this conjecture. Rödl \cite{Rodl} proved the case $\ell=3$, and Petite, Tardos, and Walczak \cite{PTW} showed that $f(k,4)$, it it exists, must be lower bounded by a power tower of height $k$.
Besides, Steiner \cite{Steiner} was able to prove the conjecture when replacing \emph{girth} by \emph{odd-girth} (minimum size of an odd cycle), and Li \cite{Li} proved it for graphs with few edges relative to their chromatic number.

\medskip

To our knowledge, this conjecture was never studied in the directed setting.
The \emph{dichromatic number} of a digraph $D$, denoted $\dic(D)$ is the least integer $k \ge 0$ such that the vertices of $D$ can be partitioned into $k$ acyclic sets. This parameter was introduced by \Erd~\cite{Dichro} and Neumann-Lara \cite{VNL}, and is now widely accepted as the correct generalisation of the chromatic number to the directed setting.
The \emph{digirth} of a digraph $D$, denoted $\dig(D)$, is the minimum size of a directed cycle in $D$. A digraph is a \emph{tournament} if it has exactly one arc between every pair of vertices.
We prove the following tournament analogue of the \Erd--Hajnal conjecture.

\begin{theorem}[store=main]\label{thm:main}
    For every integers $k \ge 0$ and $\ell \ge 3$, there exists an integer $f(k,\ell) \ge 0$ such that every tournament $T$ with $\dic(T) \ge f(k,\ell)$ contains a subdigraph $H \subseteq T$ with $\dic(H) > k$ and $\dig(H) > \ell$.
\end{theorem}

Note that we forbid $H$ from containing small directed cycles, but it may contain other orientations of a small cycle, such as the one form by three arcs $ab$, $bc$, and $ac$.
We further conjecture that the subdigraph $H$ can be found with no small cycles at all.
The \emph{girth} of a digraph $D$, denoted $\girth(D)$, is the girth of the underlying graph.

\begin{conj}\label{conj:EHT}
    For every integer $k \ge 0$ and $\ell \ge 3$, there exists an integer $f(k,\ell) \ge 0$ such that every tournament $T$ with $\dic(T) \ge f(k,\ell)$ contains a subdigraph $H \subseteq T$ with $\dic(H) > k$ and $\girth(H) > \ell$.
\end{conj}

The class of tournaments is already quite rich, as witnessed by work such as \cite{AJJ, Moon, NSS}. In fact, we show that Conjecture \ref{conj:EHT} is at least as strong as Conjecture \ref{conj:EH}.

\begin{theorem}[store=bis]\label{thm:bis}
    Conjecture \ref{conj:EHT} implies Conjecture \ref{conj:EH}.
\end{theorem}

In \Cref{sec:preli}, we introduce all the definitions, notations, and results needed for our proofs.
In \Cref{sec:main}, we prove \Cref{thm:main}. Lastly in \Cref{sec:main}, we prove \Cref{thm:bis} and show that Conjecture \ref{conj:EHT} holds for a specific family of tournament.

\section{Preliminaries}\label{sec:preli}

We refer the reader to classical textbooks such as \cite{BondyMurty} for any undefined terminology.
Given an integer $n \ge 0$, we let $[n]$ denote the set $\{1, \dots, n\}$.
In this paper, an ordering of a graph/digraph means a total ordering of its vertices.

\paragraph{Digraphs}
Let $D$ be a digraph. If $xy \in A(D)$, we say that $x$ is an in \emph{in-neighbour} of $y$, and $y$ an \emph{out-neighbour} of $x$. Given a vertex $x$, we let $N^+(x)$ and $N^-(x)$ denote respectively the set of out-neighbours and the set in-neighbours of $x$.
For two disjoint sets of vertices $X$ and $Y$, we write $X \Rightarrow Y$ to say that $xy \in A(D)$ for all $x \in X$ and all $y \in Y$. For simplicity, when $X = \{x\}$, we write $x \Ra Y$ in place of $\{x\} \Ra Y$.

We also use the symbol $\Ra$ to denote a composition  operation on digraphs: given two digraphs $D_1$ and $D_2$, we let $D_1\Ra D_2$ denote the digraph obtained from the disjoint union of $D_1$ and $D_2$ by adding all arcs from $V(D_1)$ to $V(D_2)$. Lastly, we denote by $D[X]$ the subdigraph of $D$ induced by the set of vertices $X$.

\paragraph{Backedge graphs}
Let $D$ be a digraph and $\prec$ an ordering of $D$. Given two disjoint subset of vertices $A,B \subseteq V(D)$, we write $A \prec B$ to say that $a \prec b$ for all $a \in A$ and all $b \in B$. An arc $uv \in A(D)$ is called \emph{forward} if $u \prec v$, and \emph{backward} if $v \prec u$. The \emph{backedge graph} $D^{\prec}$ of $D$ with respect to $\prec$ is the undirected graph on the vertex set $V(D)$ with an edge $uv$ whenever there is a backward arc between $u$ and $v$.

\paragraph{Dicolouring and directed clique number}
Let $D$ be a digraph and $k \ge 0$ an integer. A \emph{$k$-dicolouring} of $D$ is a partition of its vertices into $k$ subsets each inducing an acyclic subdigraph. The \emph{dichromatic number} of $D$, denoted by $\dic(D)$, is the least integer $k\ge 0$ such that $D$ has a $k$-dicolouring. It is folklore that
\begin{equation}\label{eq:dic}
    \dic(D) = \min \big\{ \chr(D^\prec) \colon \prec \text{ ordering of } D \big\}.
\end{equation}
We define the \emph{directed clique number} $\diomega(D)$ of $D$ as :
\[
    \diomega(D) = \min \big\{ \ome(D^\prec) \colon \prec \text{ ordering of } D \big\}.
\]
For a subset of vertices $X \subseteq V(D)$, when $D$ is clear from the context, we simply write $\dic(X)$ and $\diomega(X)$ in place of $\dic(D[X])$ and $\diomega(D[X])$ respectively.
An ordering $\prec$ of $D$ is called a $\diomega$-ordering if $\omega(D^{\prec})= \diomega(D)$, and a $\dic$-ordering if $\chr(D^{\prec}) = \dic(D)$.

\paragraph{Tournaments}
A \emph{tournament} is an orientation of a complete graph, that is, an oriented graph with exactly one arc between every pair of vertices. 
Given three tournaments $T_1, T_2, T_3$, we denote by $\Delta(T_1,T_2,T_3)$ the tournament obtained from  disjoint copies of $T_1, T_2, T_3$ by adding arcs between them so that $T_1 \Ra T_2 \Ra T_3 \Ra T_1$.

Let $A_0$ and $D_0$ denote one-vertex tournaments, and for every integer $n \ge 1$, let $D_n = \Delta(D_0, D_{n-1}, D_{n-1})$, and denote $A_n$ the tournament made from $n$ copies $T_1, \dots, T_n$ of $A_{n-1}$ and $n+1$ additional vertices $v_0, \dots, v_n$ such that for $0 \le i < j \le n$, we have $v_i \La v_j$ and $v_i \Ra T_j$, and for $1 \le i < j \le n$, we have $T_i \Ra v_j$ and $T_i \Ra T_j$.

\begin{figure}[h]
    \centering
    \begin{tikzpicture}
            \vertex (v) at (0,.75) {} ;
            \node at (0,0) {$\Ra$} ;
            \node[rotate=45] at (-.25,.5) {$\La$} ;
            \node[rotate=-45] at (.25,.5) {$\La$} ;
            \draw[blue, fill=blue!20] (-.75,0) circle (.5) node[black] {$D_{n-1}$} ;
            \draw[blue, fill=blue!20] (.75,0) circle (.5) node[black] {$D_{n-1}$} ;

            \foreach \i in {0,...,3}
                \vertex (v_\i) at (3+2.5*\i,0) {$v_\i$} ;
            \foreach \i in {1,...,3} {
                \vertex[blue, fill=blue!20] (T_\i) at (1.75+2.5*\i,0) {\color{black}$A_{n-1}$} ;
                \node at (2.5*\i+1,0) {$\Ra$} ;
                \node at (2.5*\i+2.5,0) {$\Ra$} ;
            }
            \draw[-latex, red] (v_1) to[bend right=45] (v_0) ;
            \draw[-latex, red] (v_2) to[bend right=45] (v_0) ;
            \draw[-latex, red] (v_3) to[bend right=45] (v_0) ;
            \draw[-latex, red] (v_2) to[bend right=45] (v_1) ;
            \draw[-latex, red] (v_3) to[bend right=45] (v_1) ;
            \draw[-latex, red] (v_3) to[bend right=45] (v_2) ;
            \draw[-{Implies}, double distance=2pt] (T_1) to[bend right] (T_2) ;
            \draw[-{Implies}, double distance=2pt] (T_1) to[bend right] (T_3) ;
            \draw[-{Implies}, double distance=2pt] (T_2) to[bend right] (T_3) ;
    \end{tikzpicture}
    \caption{Tournaments $D_n$ and $A_n$}
    \label{fig:An_Dn}
\end{figure}
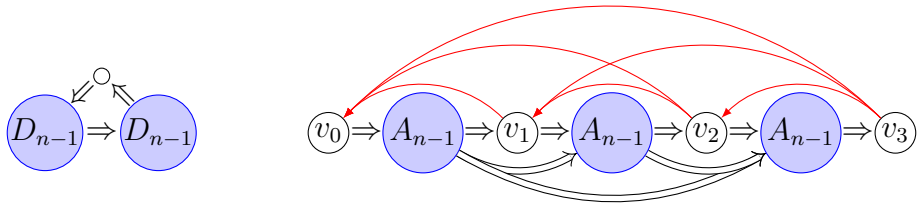

These two families of tournaments are at the center of the next theorem which characterizes tournaments with large directed clique number.

\begin{theorem}[Crew, Fan, Koerts, Moore, \& Spirkl \cite{Crew}]\label{thm:An_Dn}
    For every integer $n \ge 0$, there exists an integer $h_n \ge 0$ such that every tournament $T$ with $\diomega(T) \ge h_n$ contains a subtournament isomorphic to $A_n$ or $D_n$.
\end{theorem}

Thus, a tournament either contains a copy $A_n$ or $D_n$ for some large integer $n \ge 0$, or it has bounded directed clique number. 
We prove \Cref{thm:main} by dealing with these three cases separately in the three lemmas that follow.

\section{Proof of \Cref{thm:main}}\label{sec:main}

First, we show that the tournaments $A_n$ satisfies \Cref{thm:main}.

\begin{lemma}\label{lemma:An}
    For all integers $k \ge 0$ and $\ell \ge 3$, the tournament $A_{k\ell}$ contains a subdigraph $H \subseteq A_{k\ell}$ with $\dic(H) > k$ and $\dig(H) > \ell$.
\end{lemma}

\begin{proof}
    We proceed by induction on $k$. When $k=0$, we let $H = A_0$ be a one-vertex digraph.
    Let $k \ge 1$ and suppose the tournament $A_{(k-1)\ell}$ contains a subdigraph $H$ with $\dic(H) > k-1$ and $\dig(H) > \ell$.
    The tournament $A_{k\ell}$ is made from $k\ell$ copies $T_1, \dots, T_{k\ell}$ of $A_{k\ell-1}$ and $k\ell+1$ additional vertices $v_0, \dots, v_{k\ell}$ such that for $0 \le i < j \le n$, we have $v_i \La v_j$ and $v_i \Ra T_j$, and for $1 \le i < j \le n$, we have $T_i \Ra v_j$ and $T_i \Ra T_j$.

    As $(k-1)\ell \le k\ell-1$, each copy $T_i$ of $A_{k\ell-1}$ contains a copy of $A_{(k-1)\ell}$ and thus a subdigraph $H_i \subseteq T_i$ such that $\dic(H_i) > k-1$ and $\dig(H_i) > \ell$, using the induction hypothesis. Let $H$ be the subdigraph of $A_{k \ell}$ induced by every such $H_i$ and the $k+1$ vertices $v_0,v_\ell,\dots,v_{k\ell}$, by further removing the arcs between $H_i$ and $H_j$ when $|i-j| \ne 1$, as well as the arcs between $H_i$ and $v_j$ when $i \notin \{j,j+1\}$ (see \Cref{fig:An}).

    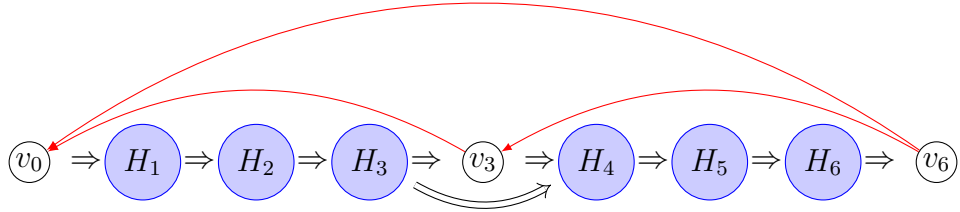
\begin{figure}[h]
        \centering
        \begin{tikzpicture}
            \foreach \i in {0,3,6}
                \vertex (v_\i) at (2*\i,0) {$v_\i$} ;
            \foreach \i in {1.5,3,4.5,7.5,9,10.5}
                \draw[blue, fill=blue!20] (\i,0) circle (.5) ;
            \foreach \i in {1,2,3}
                \node (H_\i) at (1.5*\i,0) {$H_\i$} ;
            \foreach \i in {4,5,6}
                \node (H_\i) at (1.5*\i+1.5,0) {$H_\i$} ;
            \foreach \i in {0,...,7}
                \node at (1.5*\i+.75,0) {$\Ra$} ;
            \draw[-latex, red] (v_3) to[bend right] (v_0) ;
            \draw[-latex, red] (v_6) to[bend right] (v_3) ;
            \draw[-latex, red] (v_6) to[bend right=35] (v_0) ;
            \draw[-{Implies}, double distance=2pt] (5.1,-.3) to[bend right] (6.9,-.3) ;
        \end{tikzpicture}
        \caption{Construction of $H\subseteq A_{k\ell}$ (here $k=2$ and $\ell=3$)}
        \label{fig:An}
    \end{figure}

    We claim that this subdigraph $H \subseteq A_{k\ell}$ verifies $\dic(H) > k$ and $\dig(H) > \ell$. Suppose $H$ has a $k$-dicoloring. By the pigeonhole principle, among the $k+1$ vertices $v_0, v_\ell \dots, v_{k\ell}$ two have the same color $c$, say $v_i$ and $v_j$ with $i<j$. For each $i < t \le j$, as $\dic(H_t) > k-1$, there is a vertex $u_t \in H_t$ with colour $c$. Then, $v_i u_{i+1} \cdots u_jv_j$ form a monochromatic directed cycle, a contradiction. Hence $\dic(H) > k$.
    
    Let $C$ be a directed cycle in $H$.
    If $C \subseteq H_i$ for some $i \in [k\ell]$, then $|C| > \ell$ as $\dig(H_i) > \ell$.
    Otherwise the cycle $C$ intersects at least two of the subdigraphs $H_i$ or the vertices $v_i$.
    Observe that every path from $H_j$ or $v_j$ to $H_i$ or $v_i$ with $i < j$, or from $H_i$ to $v_i$, must use a backward arc of the form $v_{p\ell}v_{q\ell}$.
    The cycle $C$ must also go through the $\ell$ subdigraphs $H_{p\ell+1}, \dots, H_{q\ell}$, thus $|C| \ge \ell+2$.
\end{proof}

Next, we show that the tournaments $D_n$ verify \Cref{thm:main}.

\begin{lemma}\label{lemma:Dn}
    For all integers $k \ge 0$ and $\ell \ge 1$, the tournament $D_{k\ell}$ contains a subdigraph $H \subseteq D_{k\ell}$ with $\dic(H) > k$ and $\dig(H) > 2^\ell$.
\end{lemma}

\begin{proof}
     We proceed by induction on $k$. When $k=0$, we let $H = D_0$ be a one-vertex digraph.
     Let $k \ge 1$ and suppose the tournament $D_{(k-1)\ell}$ contains a subdigraph $H$ with $\dic(H) > k-1$ and $\dig(H) > \ell$.
    The tournament $D_{k\ell}$ is made of a vertex $v$ and two copies $D'$ and $D''$ of $D_{k\ell-1}$ verifying $v \Ra D' \Ra D'' \Ra v$.

    For every $p,q \ge0$, the tournament $D_{p+q}$ contains $2^q$ copies $T_1,\dots,T_{2^q}$ of $D_p$ such that $T_i \Ra T_j$ for all $i<j$ (proved via a trivial induction on $q$).
    In particular, there are copies $T_1, \dots, T_{2^{\ell-1}} \subseteq D'$ and $T_{2^{\ell-1}+1}, \dots, T_{2^\ell} \subseteq D''$ of the tournament $D_{(k-1)\ell}$ satisfying $T_i \Ra T_j$ for all $i<j$ (we use that $D' \Ra D''$).
    
    By induction hypothesis, each copy $T_i$ of $D_{(k-1)\ell}$ contain a subdigraph $H_i \subseteq T_i$ such that $\dic(H_i) > k-1$ and $\dig(H) > \ell$.
    Let $H$ be the digraph induced by every such $H_i$ and the vertex $v$, by further removing the arcs between $v$ and $H_i$ when $i \notin \{1,2^\ell\}$, as well as the arcs between $H_1$ and $H_j$ when $|i-j| \ne 1$ (see \Cref{fig:Dn}).
    
    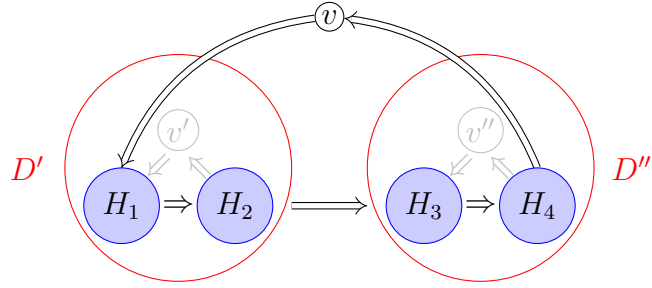
\begin{figure}[h]
        \centering
        \begin{tikzpicture}
            \draw[red] (-2,0) circle (1.5) ;
            \draw[red] (2,0) circle (1.5) ;
            \node[red] at (-4,0) {$D'$} ;
            \node[red] at (4,0) {$D''$} ;
            
            \vertex (v) at (0,2) {$v$} ;
            \vertex[lightgray] (v') at (-2,.5) {$v'$} ;
            \vertex[lightgray] (v'') at (2,.5) {$v''$} ;
            \node[lightgray, rotate=45] at (-2.25,0) {$\La$} ;
            \node[lightgray, rotate=-45] at (-1.75,0) {$\La$} ;
            \node[lightgray, rotate=45] at (1.75,0) {$\La$} ;
            \node[lightgray, rotate=-45] at (2.25,0) {$\La$} ;
            \node at (-2,-.5) {$\Ra$} ;
            \node at (2,-.5) {$\Ra$} ;
            \draw[-{Implies}, double distance=2pt] (-.5,-.5) to (.5,-.5) ;
            \draw[-{Implies}, double distance=2pt] (v) to[bend right] (-2.75,0) ;
            \draw[-{Implies}, double distance=2pt] (2.75,0) to[bend right] (v) ;
            
            \draw[blue, fill=blue!20] (-2.75,-.5) circle (.5) node[black] {$H_1$} ;
            \draw[blue, fill=blue!20] (-1.25,-.5) circle (.5) node[black] {$H_2$} ;
            \draw[blue, fill=blue!20] (1.25,-.5) circle (.5) node[black] {$H_3$} ;
            \draw[blue, fill=blue!20] (2.75,-.5) circle (.5) node[black] {$H_4$} ;
        \end{tikzpicture}
        \caption{Construction of $H$ (here $\ell=2$)}
        \label{fig:Dn}
    \end{figure}

    We claim that this subdigraph $H \subseteq D_{k\ell}$ verifies $\dic(H) > k$ and $\dig(H) > 2^\ell$. Suppose $H$ has a $k$-dicoloring, and let $c$ denote the color of the vertex $v$. For each $i \in [2^\ell]$, as $\dic(H_i) > k-1$, there is a vertex $u_i \in H_i$ with colour $c$. Then, $v u_1 \cdots u_{2^\ell}$ form a monochromatic directed cycle, a contradiction. Hence $\dic(H) > k$.

    Let $C$ be a directed cycle in $H$.
    If $C \subseteq H_i$ for some $i \in [2^\ell]$, then $|C| > 2^ \ell$ as $\dig(H_i) > \ell$.
    Observe that if otherwise $C$ is not included in one subdigraph $H_i$, then it must go through the vertex $v$ and all subdigraphs $H_i$.
    Hence $|C| \ge 2^\ell+1$.
\end{proof}

We now prove \Cref{thm:main} for tournaments with bounded directed clique number.

\begin{lemma}\label{lemma:main}
    For all integers $h,k,\ell,c \ge 0$, there exists an integer $g(h,k,\ell,c) \ge 0$ such that every tournament $T$ with $\dic(T) \ge g(h,k,\ell,c)$ and $\diomega(T) \le h$ contains a subdigraph $H \subset T$ with $\dig(H) > \ell$ and such that for every subset of vertices $X \subset V(T)$ with $\dic(T[X]) \le c$, we have $\dic(H - X) > k$.
\end{lemma}

\begin{proof}
    We proceed by induction on $h+k$. When $h=0$, the fact that $\diomega(T) \le 0$ implies that $T$ is empty, so it suffices to set $g(h,k,\ell,c)=1$. When $k=0$, we set $g(h,k,\ell,c)=c+1$, and take $H=(V(T),\emptyset)$. Now suppose $h,k \ge 1$ and set
    \[
         g' := g(h-1,k,\ell,c),\quad g'' := g \big( h, k-1, \ell, 2g'+c \big), \quad\text{and}\quad g(h,k,\ell,c) := \ell k g'' + c+1.
    \]
    
    Fix a $\diomega$-ordering $\prec$ of $T$. For each vertex $v \in V(T)$, as $\omega(T^\prec)=h$, the neighborhood of $v$ in $T^\prec$ contains no $h$-clique, so $\diomega(N_{T^\prec}(v)) \le h-1$. By induction hypothesis, if $\dic(N_{T^\prec}(v)) \ge g' = g(h-1,k,\ell,c)$ then $N_{T^\prec}(v)$ contains a subdigraph $H$ satisfying the desired result. Therefore, we may suppose that
    \[
        \dic(N_{T^\prec}(v))< g' \qquad \text{ for all } v \in V(T).
    \]

    We construct bags of vertices $B_0, B_1, \dots$ by adding vertices to $B_0$, in the order given by $\prec$, until $\dic(B_0) = g''$, then adding the following vertices to $B_1$ until $\dic(B_1) = g''$, and so on until no more vertices remain. We obtain a partition $B_0 \prec \cdots \prec B_t$ of the vertices such that $\dic(B_i) = g''$ for all $0 \le i < t$, and $\dic(B_t) \le g''$.

    For each $0 \le i < t$, as $\dic(B_i) = g'' = g \big( h, k-1, \ell, 2g'+c \big)$, by induction hypothesis, the tournament $T[B_i]$ contain a subdigraph $H_i \subseteq T[B_i]$ with $\dig(H_i) > \ell$ and such that for every subset of vertices $X \subset V(T)$ with $\dic(X) \le 2g'+c$, we have $\dic(H_i - X) > k-1$.
    Let $H$ be the digraph made from the subgraphs $H_i$ for $0 < i < t$, every (forward) arc from $B_i$ to $B_{i+1}$ for $0 \le i < t$, and all (backward) arcs from $B_j$ to $B_i$ for $i<j$ and $i \equiv j \pmod \ell$. In what follows, we refer to this last type of arcs as \emph{long} arcs.

    We claim that $H$ satisfies the desired result.
    The proof that $\dig(H) > \ell$ is basically the same as in \Cref{lemma:An}.
    Let $X \subset V(T)$ be a subset of vertices with $\dic(T[X])\le c$, and toward a contradiction, suppose that $H-X$ has a $k$-dicoloring.
    We distinguish two cases depending on whether there is a monochromatic long arc.
    
    \medskip

    If there is a monochromatic long arc, then we claim that there is a monochromatic directed cycle. 
    First, note that for $v_{t-1} \in B_{t-1}$ and $v_{t+1} \in B_{t+1}$, as $v_{t-1} \prec B_t \prec v_{t+1}$, every in-neighbor of $v_{t-1}$ in $B_t$ is adjacent to $v_{t-1}$ in the backedge graph $T^\prec$, that is $(N^-(v_{t-1}) \cap B_t) \subseteq N_{T^\prec}(v_{t-1})$, and likewise we have $(N^+(v_{t+1}) \cap B_t) \subseteq N_{T^\prec}(v_{t+1})$. Recall that $\dic(N_{T^\prec}(v)) \le g'$ for all vertices $v \in V(T)$, so $\dic(X') \le 2g'+c$ where $X' = (N^-(v_{t-1}) \cap B_t) \cup (N^+(v_{t+1}) \cap B_t) \cup X$, and thus $\dic(H_t-X') > k-1$.
    
    Now let $u_ju_i$ be a long arc with $u_i \in B_i$ and $u_j \in B_j$ ($i<j$) both having color $\alpha$. We iteratively build a monochromatic path $u_i u_{i+1} \cdots u_{j-2}$. For $i < t < j-1$, denote $X_t = (N^-(u_{t-1}) \cap B_t) \cup X$, then as $\dic(H_t - X_t) > k-1$, there exists a vertex $u_t \in H_t - X_t$ with color $\alpha$. Last, denote $X_{j-1} = (N^-(u_{j-2}) \cap B_{j-1}) \cup (N^+(u_{j}) \cap B_{j-1}) \cup X$, then as $\dic(H_{j-1} - X_{j-1}) > k-1$, there exists a vertex $u_{j-1} \in H_{j-1} - X_{j-1}$ with color $\alpha$. Now $u_i \cdots u_j$ is a monochromatic cycle in $H-X$, a contradiction.

    \medskip

    If there is no monochromatic long arc, then we claim that $\dic(T-X) \le \ell k g''$ which contradicts $\dic(T) \ge \ell k g''+c+1$ as $\dic(T[X]) \le c$.
    We partition the bags in $\ell$ classes according to their index modulo $\ell$, that is, for each $i \in [\ell]$, we denote $C_i = \bigcup_{i \equiv j \pmod{\ell}} B_j$. Observe that every backward arc between two different bags of $C_i$ is a long arc and thus bichromatic. For each index $i \in [\ell]$ and color $\alpha \in [k]$, let $T_{i,k}$ be the subtournament of $T-X$ induced by all vertices of color $\alpha$ in $C_i$. Note that there is no arc in $T_{i,k,}$ from $B_j$ to $B_{j'}$ with $j>j'$, so we have $\dic(T_{i,k}) \le \max_j \dic(B_j) = g''$. It follows that $\dic(T-X) \le \ell k g''$, a contradiction. Therefore we have $\dic(H-X) \ge k$.
\end{proof}

We may finally deduce \Cref{thm:main}.

\getkeytheorem{main}

\begin{proof}
    Let $T$ be a tournament with $\dic(T) \ge g(h_{k\ell}, k, \ell, 0)$, where $h_{k\ell} \ge 0$ is defined by \Cref{thm:An_Dn}, and $g(h_{k\ell}, k, \ell, 0) \ge 0$ is defined by \Cref{lemma:main}.
    
    If $\diomega(T) \le h_{k\ell}$, then by \Cref{lemma:main}, the tournament $T$ contains a subdigraph $H \subseteq T$ with $\dig(H) > \ell$ and $\dic(H) > k$.
    If $\diomega(T) > h_{k\ell}$, then by \Cref{thm:An_Dn}, the tournament $T$ contains a subtournament $S \subseteq T$ isomorphic to $A_{k\ell}$ or $D_{k\ell}$ .
    Using \Cref{lemma:An} or \Cref{lemma:Dn}, we can find a subdigraph $H \subseteq S$ with $\dig(H) > \ell$ and $\dic(H) > k$.
\end{proof}

\section{On Conjecture \ref{conj:EHT}} \label{sec:conj}

We now prove \Cref{thm:bis}, and then show that the family of tournaments $(A_n)$ verifies Conjecture \ref{conj:EHT}. For the former, we need the following useful inequality.

\begin{prop}[Nguyen, Scott, \& Seymour \cite{NSS}]\label{prop:ineq}
    Let $T$ be a tournament. For every ordering $\prec$ of $T$, we have:
    \[
        \chr(T^\prec) \le \dic(T) \cdot \ome(T^\prec).
    \]
\end{prop}

\getkeytheorem{bis}

\begin{proof}
    Let $k\ge 0$ and $\ell \ge 3$.
    Suppose there exists an integer $f(k,\ell) \ge 0$ that satisfies Conjecture \ref{conj:EHT}.
    Fix a graph $H_{k,\ell}$ with $\girth(H_{k,\ell}) > \ell$ and $\chr(H_{k,\ell}) > k$, which exists by \cite{Erdos}, and let $G$ be a graph with $\chr(G) \ge |V(H_{k,\ell})| \cdot f(k,\ell)$. If $G$ contains a $|V(H_{k,\ell})|$-clique, then $H_{k,\ell}$ is a subgraph of $G$ and we're done, so $\omega(G) < |V(H_{k,\ell})|$.
    
    Let $T$ be a tournament that admits an ordering $\prec$ such that $T^\prec = G$. By Proposition~\ref{prop:ineq}, we have $\chr(G) \le \dic(T) \cdot \omega(G)$ which implies $\dic(T) \ge f(k,\ell)$. By our assumption, the tournament $T$ contains a subdigraph $S \subseteq T$ such that $\girth(S) > \ell$ and $\dic(S) > k$. Then, $S^\prec$ is a subgraph of $T^\prec = G$ with $\girth(S^\prec) > \ell$ and $\chr(S^\prec) > k$.
\end{proof}

In support of Conjecture \ref{conj:EHT}, we prove that it holds for the family $(A_n)$. We use the existence of hypergraphs with large uniformity, chromatic number, and girth.

Let $H$ be an hypergraph and $r \ge 2$ an integer. We say that $H$ is \emph{$r$-uniform} if every hyperedge $e \in E(H)$ has size $r$. The \emph{chromatic number} of $H$, denoted $\chr(H)$, is the least integer $k \ge 0$ such that the vertices of $H$ can be partitioned in $k$ sets, none of which contains an hyperedge. A \emph{closed trail} of size $k$ in $H$ is an alternating sequence $v_1e_1\cdots v_ke_k$ of distinct vertices $v_i \in V(H)$ such that for every $i \in [k]$, we have $v_i \in e_i$ and $v_{i+1} \in e_i$ (with indices taken modulo $k$). If furthermore all edges are distinct in the sequence, then the closed trail is a \emph{cycle}. The \emph{girth} of $H$, denoted $\girth(H)$, is the minimum size of a cycle in $H$. 

\begin{theorem}[\Lov~\cite{Lovasz}]\label{thm:Lov}
    For every integers $k \ge 0$ and $\ell,r \ge 2$, there exists an $r$-uniform hypergraph $H$ with $\chr(H) > k$ and $\girth(H) > \ell$.
\end{theorem}

\begin{prop}
    For every integers $k \ge 0$ and $\ell \ge 3$, there exists an integer $f(k,\ell) \ge 0$ such that $A_{f(k,\ell)}$ contains a subdigraph $S \subseteq A_{f(k,\ell)}$ with $\dic(S) > k$ and $\girth(S) > \ell$.
\end{prop}

\begin{proof}
    We proceed by induction on $k$. When $k=0$, we let $f(k,\ell)=0$ and take $S = A_0$.
    Let $k \ge 1$ and suppose $f(k-1,\ell)$ is well-defined.
    Denote $r = |V(A_{f(k-1,\ell)})|$ and let $H$ be a $(r\ell)$-uniform hypergraph $H$ with $\chr(H) > k$ and $\girth(H) > \ell$, which exists by \Cref{thm:Lov}.
    Denote $n=|V(H)|$ and $m=|E(H)|$ and set $f(k,\ell) = n \cdot m$.
    
    The tournament $A_{f(k,\ell)}$ is made of $f(k,\ell)$ copies $T_1, \dots, T_{f(k,\ell)}$ of $A_{f(k,\ell)-1}$ and $f(k,\ell)+1$ vertices $v_0, \dots, v_{f(k,\ell)}$ such that for $0 \le i < j \le f(k,\ell)$, we have $v_i \La v_j$ and $v_i \Ra T_j$, and for $1 \le i < j \le f(k,\ell)$, we have $T_i \Ra v_j$ and $T_i \Ra T_j$. As $f(k,\ell)-1 \ge f(k-1,\ell)$, each copy $T_i$ of $A_{f(k,\ell)-1}$ contains a copy of $A_{f(k-1,\ell)}$, and thus a subdigraph $S_i \subseteq T_i$ with $\dic(S_i) > k-1$ and $\girth(S_i) > \ell$. Without loss of generality we assume that $|V(S_i)|=r=|V(A_{f(k-1,\ell)})|$ for every $S_i$ (by keeping isolated vertices).

    Identify the $n$ vertices $v_{m}, v_{2m}, \dots, v_{nm}$ of $A_{f(k,\ell)}$ with the vertices $u_1, \dots, u_{n}$  of $H$. Denote $E(H)= \{e_1, \dots, e_{m}\}$.
    To each hyperedge $e_i = \{u_{i_1}, \dots, u_{i_{r\ell}} \}$ with $i_1 < \cdots < i_{r\ell}$, we associate a subdigraph $S'_i = S_{i_r \cdot m+i}$.
    Note that each hyperedge is associated to different subdigraph, and that we have $\{ u_{i_1}, \dots, u_{i_{r}} \} \La S_i' \La \{ u_{i_{r(\ell-1)+1}}, \dots, u_{i_{r\ell}} \}$.

    We build a subdigraph $S \subseteq A_{f(k,\ell)}$ from the subdigraphs $S_i'$ and the vertices $u_i$ by adding arcs as follow. For each hyperedge $e_i = \{u_{i_1}, \dots, u_{i_{r\ell}} \}$ associated to a subdigraph $S_i'$ with vertex set $V(S_i') = \{w^i_1, \dots, w_r^i\}$, we add to $S$ a set of arcs forming $r$ disjoint directed cycles of size $\ell+1$, each cycle intersecting $S_i'$ in exactly one vertex. More precisely, for each $j \in [r]$, we keep the cycle $C_i^j = w^i_j u_{i_{r(\ell-1)+j}} u_{i_{r(\ell-2)+j}} \cdots u_{i_{j}}$.
    
    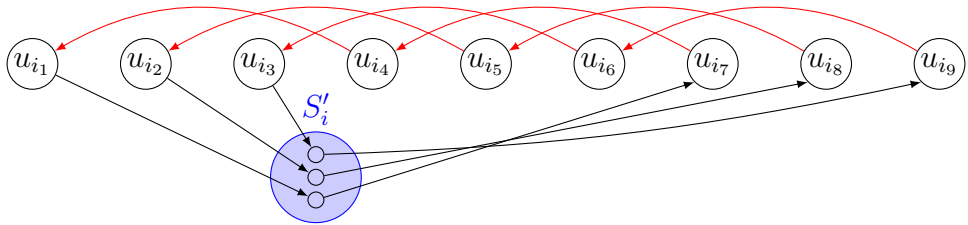
\begin{figure}[h]
        \centering
        \begin{tikzpicture}
            \foreach \i in {1,...,9}
                \vertex (u_\i) at (\i*1.5,0) {$u_{i_{\i}}$} ;

            \draw[blue, fill=blue!20] (5.25,-1.5) circle (.6)
                node (S_i) at (5.25,-.6) {$S_i'$} ;
            \vertex (w_1) at (5.25,-1.8) {} ;
            \vertex (w_2) at (5.25,-1.5) {} ;
            \vertex (w_3) at (5.25,-1.2) {} ;
            
            \draw[-latex, red] (u_9) to[bend right] (u_6) ;
            \draw[-latex, red] (u_8) to[bend right] (u_5) ;
            \draw[-latex, red] (u_7) to[bend right] (u_4) ;
            \draw[-latex, red] (u_6) to[bend right] (u_3) ;
            \draw[-latex, red] (u_5) to[bend right] (u_2) ;
            \draw[-latex, red] (u_4) to[bend right] (u_1) ;

            \draw[-latex] (u_1) to (w_1) ;
            \draw[-latex] (u_2) to (w_2) ;
            \draw[-latex] (u_3) to (w_3) ;
            \draw[latex-] (u_7.south west) to (w_1) ;
            \draw[latex-] (u_8.south west) to (w_2) ;
            \draw[latex-] (u_9.south west) to[bend left=5] (w_3) ;
        \end{tikzpicture}
        \caption{Construction of $S \subseteq A_{f(k,\ell)}$ for one hyperedge of $H$ (here $\ell=3$ and $r=3$)}
        \label{fig:An_bis}
    \end{figure}
    
    We claim that this subdigraph $S \subseteq A_{f(k,\ell)}$ satisfies $\dic(S) > k$ and $\girth(S) > \ell$. Suppose $S$ has a $k$-dicoloring. As the hypergraph $H$ is not $k$-colorable, there exists an hyperedge $e_i = \{u_{i_1}, \dots, u_{i_{r\ell}} \}$ whose vertices all have the same color $c$. Consider the associated subdigraph $S'_i$. As $\dic(S_i') > k-1$, there is a vertex $w^i_j \in S_i'$ of color $c$. Then $C^i_j$ is a monochromatic directed cycle, a contradiction. Henceforth $\dic(S) > k$.

    Let $C$ be a (not necessarily directed) cycle $C$ in $S$. Consider the hypergraph $H'$ with vertex set $V(S)$ and hyperedges $\{ e_i \cup S_i' : i \in [m]\}$. As the subdigraphs $S_i'$ are disjoint, we have $\girth(H') = \girth(H)$. Now every arc of $S$ is included in an hyperedge of $H'$, hence the cycle $C$ induces a closed trail in $H'$. If this trail is trivial (only one edge appears in the sequence defining the trail), that is $C$ is included in $S_i' \cup e_i$ for some $i \in [m]$, then either $C$ in contained in $S_i'$ and $|C| \ge \girth(S_i') > \ell$, or $C$ contains a cycle of the form $C^i_j$ and $|C| \ge |C^i_j| = \ell+1$. Otherwise the trail contains a non-trivial cycle of $H'$, thus $|C| \ge \girth(H') > \ell$.

\end{proof}


\paragraph*{Acknowledgement}

We thank Pierre Aboulker for early discussions on this problem.

\paragraph*{Use of AI}

The authors declare no significant use of AI in the creation of this work.


\bibliographystyle{plain}
\bibliography{biblio}

\end{document}

%% file: macros.tex
\usepackage[T1]{fontenc}
\usepackage[utf8]{inputenc}
\usepackage[margin=3cm]{geometry}

\usepackage{amsmath, amsfonts, amssymb,amsthm}
\usepackage[overload]{keytheorems}

\usepackage{hyperref, url, cleveref}
\usepackage{tikz, xcolor}
\usepackage{authblk}

\newcommand{\Ra}{\Rightarrow}

\newcommand{\La}{\Leftarrow}

\newcommand{\ora}[1]{\overrightarrow{#1}}

\DeclareMathOperator{\chr}{\chi}
\DeclareMathOperator{\ome}{\omega}

\DeclareMathOperator{\girth}{girth}
\DeclareMathOperator{\dic}{\ora \chi}
\DeclareMathOperator{\diomega}{\ora \omega}

\DeclareMathOperator{\dig}{\ora\girth}

\newcommand{\Erd}{Erd\H os}

\newcommand{\Lov}{Lov\'asz}

\renewcommand{\emptyset}{\varnothing}

\renewcommand{\phi}{\varphi}

\let\le\leqslant
\let\ge\geqslant

\newtheorem{theorem}{Theorem}
\newtheorem{lemma}[theorem]{Lemma}
\newtheorem{conj}[theorem]{Conjecture}
\newtheorem{prop}[theorem]{Proposition}

	{\noindent {\it Proof of Claim:
    }}
	{\hfill $\square$ \par\vspace{11pt}}

\usetikzlibrary{arrows.meta}
\tikzstyle{vertex}=[circle, draw, inner sep=1pt, minimum size=6pt]
\newcommand{\vertex}{\node[vertex]}